\documentclass{amsart}

\usepackage{amssymb}
\usepackage{amsmath}
\usepackage{amsfonts}

\usepackage{graphicx}
\usepackage{xy, color, hyperref, ytableau}
\usepackage{todo}
\xyoption{all}
\usepackage{listings,verbatim}
\lstdefinelanguage{Sage}[]{Python}
{morekeywords={False,sage,True},sensitive=true}
\definecolor{dblackcolor}{rgb}{0.0,0.0,0.0}
\definecolor{dbluecolor}{rgb}{0.01,0.02,0.7}
\definecolor{dgreencolor}{rgb}{0.2,0.4,0.0}
\definecolor{dgraycolor}{rgb}{0.30,0.3,0.30}

\newcommand{\mb}{\mathbb}
\newcommand{\la}{\lambda}
\newcommand{\half}{\frac 12}
\newcommand{\C}{\mathbf C}
\newcommand{\core}[2]{\mathrm{core}_{#2}{#1}}

\DeclareMathOperator{\bin}{bin}

\newcommand{\eps}{\varepsilon}

\DeclareMathOperator{\sgn}{sgn}
\DeclareMathOperator{\ab}{ab}
\DeclareMathOperator{\Irr}{Irr}
\DeclareMathOperator{\Res}{Res}

\DeclareMathOperator{\ind}{ind}
\newtheorem{theorem}{Theorem}

\newtheorem{proposition}[theorem]{Proposition}

\theoremstyle{remark}
\newtheorem{remark}[theorem]{Remark}
\title[Macdonald trees and determinants of representations]
{Macdonald trees and determinants of representations for finite Coxeter groups}
\author{Arvind Ayyer}
\address{Arvind Ayyer, Department of Mathematics, Indian Institute of Science, Bengaluru 560012, India.}
\email{arvind@iisc.ac.in}

\author{Amritanshu Prasad}
\address{Amritanshu Prasad, The Institute of Mathematical Sciences (HBNI), CIT campus, Taramani, Chennai 600113, India.}
\email{amri@imsc.res.in}

\author{Steven Spallone}
\address{Steven Spallone, Indian Institute of Science Education and Research, Pashan, Pune 411008, India.}
\email{sspallone@iiserpune.ac.in}
\date{\today}

\begin{document}
\begin{abstract}
  Every irreducible odd dimensional representation of the $n$'th symmetric or hyperoctahedral group, when restricted to the $(n-1)$'th, has a unique irreducible odd-dimensional constituent.
  Furthermore, the subgraph induced by odd-dimensional representations in the Bratteli diagram of symmetric and hyperoctahedral groups is a binary tree with a simple recursive description.
  We survey the description of this tree, known as the Macdonald tree, for symmetric groups, from our earlier work.
  We describe analogous results for hyperoctahedral groups.

  A partition $\lambda$ of $n$ is said to be {\em chiral} if the corresponding irreducible representation $V_\lambda$ of $S_n$ has non-trivial determinant. 
  We review our previous results on the structure and enumeration of chiral partitions, and subsequent extension to all Coxeter groups by Ghosh and Spallone.
  Finally we show that the numbers of odd and chiral partitions track each other closely.
\end{abstract}
\maketitle
\section{Summary of results}
\label{sec:introduction}
This article is mainly a survey of the main results of three papers \cite{ayyer-prasad-spallone-odd-16} and \cite{APS-chiral} and \cite{Ghosh} concerning the application of the dyadic arithmetic of partitions to the representation theory of finite Coxeter groups.

Recall that Young's lattice is the set $\Lambda$ of integer partitions, partially ordered by containment of Young diagrams.
It has a unique minimal element $\emptyset$, the trivial partition of $0$.
Its Hasse diagram is known as Young's graph.
For each $\lambda\in \Lambda$, let $f_\lambda$ denote the number of saturated chains in $\Lambda$ from $\emptyset$ to  $\lambda$.
This number is also the number of standard tableaux of shape $\lambda$, and equals the dimension of the irreducible representation of the symmetric group associated with $\lambda$.
The numbers $f_\lambda$ can be computed by the hook-length formula of Frame, Robinson and Thrall \cite[Theorem~1]{frame1954hook}.

We say that a partition $\lambda$ is {\em odd} if $f_\lambda$ is odd.
In Ref.~\cite{ayyer-prasad-spallone-odd-16} it is shown that the subgraph induced by the subset of odd partitions in Young's graph forms an incomplete binary tree rooted at $\emptyset$. See Figure~\ref{fig:odd_in_yng} for an illustration of the first few rows of this tree.
Moreover, this binary tree has a simple recursive structure that gives a structural explanation for a formula independently due to McKay~\cite{mckay-1972} and Macdonald~\cite{macdonald1971degrees} (stated as Theorem~\ref{thm:macdonald} here) for the number of odd partitions of $n$. 
Macdonald generalized this formula to other Coxeter groups in Ref.~\cite{macdonald1973cox}. In Section~\ref{sec:odd-partitions}, we begin by explaining these results for the symmetric group. We then present our new results for the hyperoctahedral groups, which explain Macdonald's formulas.

Our results for the symmetric group have been used by Bessenrodt, Giannelli, Kleshchev, Miller, Navarro, Olsson and Tiep to propose bijective McKay correspondences for symmetric groups and study their representation-theoretic properties \cite{MR3692973,MR3687936,gagp,GM2018,MR3621673}.

The irreducible complex representations of $S_n$ are parametrized by partitions of $n$. Let $\lambda$ be a partition and $(\rho_\lambda, V_\lambda)$ be the associated irreducible representation of $S_n$.
Let $\det:GL_\C(V_\lambda)\to \C^*$ denote the determinant function.
The composition $\det\circ\rho_\lambda :S_n\to \C^*$, being a multiplicative character of $S_n$, is either the trivial character or the sign character.
We say that $\lambda$ is a \emph{chiral partition} if
$\det\circ\rho_\lambda$ is the sign character.
\begin{figure}[htbp!]
  \begin{center}
    \includegraphics[width=\textwidth]{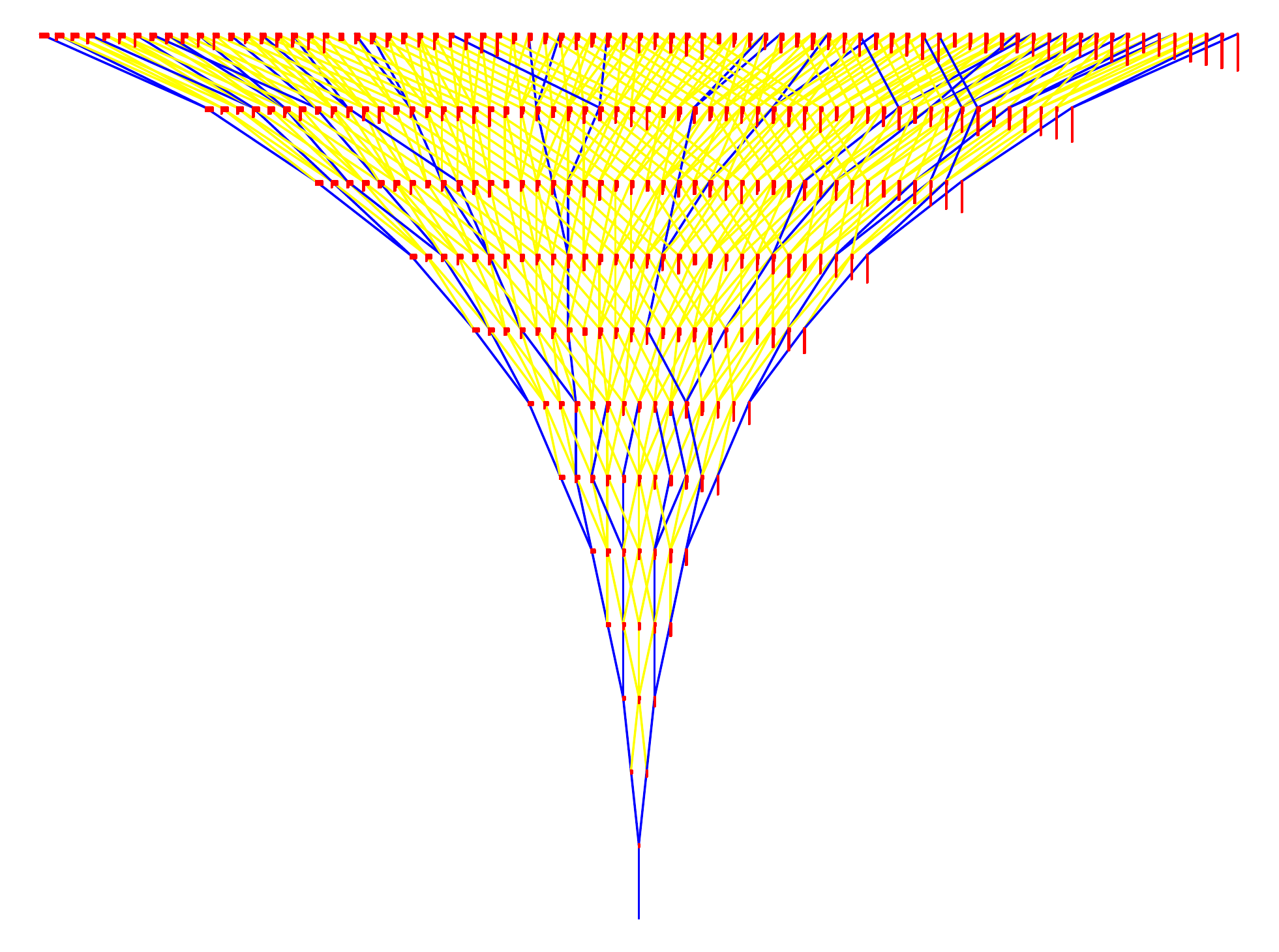}
    \caption{The first thirteen rows of the Macdonald tree (blue edges) in Young's graph (yellow edges).}
    \label{fig:odd_in_yng}
  \end{center}
\end{figure}
For each $\lambda\in \Lambda$, let $g_\lambda$ denote the number of saturated chains in $\Lambda$ from $(1,1)$ to $\lambda$, which is also
the number of standard tableaux $T$ of shape $\lambda$ such that $2$ occurs in the first column of $T$. Note that $g_{(n)} = 0$ for each $n$.
It turns out (see Section~\ref{sec:amrich-part}) that a partition $\lambda$ is chiral if and only if $g_\lambda$ is odd.
The most striking result of Ref.~\cite{APS-chiral} is a closed formula for the number of chiral partitions of $n$. This answers an open question posed in \cite[Exercise~7.55(b)]{ec2}.
In Section~\ref{sec:chiral-coxeter}, we review this result as well as analogues of chiral partitions for all Coxeter groups.
Finally, in Section~\ref{sec:comparison}, we show that the enumeration problem of odd and chiral partitions for the symmetric groups are closely related. 

We appeal to standard textbooks in representation theory (see, for example, \cite{2009JamesKerber}) for basic definitions involving partitions such as arm-length, leg-length and content.

\section{Macdonald trees}
\label{sec:odd-partitions}
Despite having nice combinatorial properties, the Bratteli diagrams of systems of Coxeter groups are objects of great complexity when viewed as directed graphs.
A powerful manifestation of this complexity is the difficulty in providing a closed formula for the number of partitions of $n$.
In this section, we shall examine the subgraphs induced in Bratteli diagrams of Coxeter groups by odd-dimensional representations.
For symmetric and hyperoctahedral groups, these subgraphs turn out to be rooted binary trees (where every node has up to two branches) whose structure can be captured in a few succinct axioms (Theorems~\ref{theorem:recursive} and~\ref{theorem:recursive-hypermac}).
We call these trees \emph{Macdonald trees} after Ian G. Macdonald, whose techniques involving cores and quotients we used. 
The structure of these trees explains the enumerative results of Macdonald~\cite{macdonald1971degrees} and McKay~\cite{mckay-1972}.
At the time we found our results, we only knew about the work of Macdonald. It was later brought to our attention that McKay had published the same results a year before.

We begin with an illustrative toy case in Section~\ref{sec:baby-case:-pascals}. We then recall our results for the symmetric group in Section~\ref{sec:odd-sym}. Our new results for the hyperoctahedral group are explained in Section~\ref{sec:odd-part-hyper}. We show that the subgraph of odd-dimensional representations for the  demihyperoctahedral group is not a tree in Section~\ref{sec:odd-demi}.

\subsection{Toy case: Pascal's triangle}
\label{sec:baby-case:-pascals}
To better appreciate the recursive structure of Macdonald trees it is instructive to begin with the much simpler example of Pascal's triangle.
Consider the graph $P$ whose nodes are pairs $(a,b)$, where $a$ and $b$ are non-negative integers.
Pairs $(a,b)$ and $(c,d)$ are connected by an edge if either $a=c$ and $|b-d|=1$ or $|a-c|=1$ and $b=d$.
This graph is the Hasse diagram of the poset $\mathbf N\times \mathbf N$, where $\mathbf N$ denotes the non-negative integers with the usual linear order.

It is easy to see that the number of number of saturated chains in $\mathbf N\times \mathbf N$ from $(0,0)$ to $(a,b)$ is the binomial coefficient $C(a,b)=\frac{(a+b)!}{a!b!}$.
Call a pair $(a,b)\in P$ an \emph{odd pair} if $C(a,b)$ is odd.
Call it a \emph{one-dimensional pair} if $C(a,b)=1$ (i.e., $a=0$ or $b=0$).
It is a now well-known discovery of Lucas (see \cite[Chapter~17]{stewart2006cut}) that the subset of $P$ consisting of odd pairs is a combinatorial version of the Sierpi\`nski gasket.
We revisit this connection from the point of view of graphs.
Let $L$ denote the subgraph induced in $P$ by odd pairs.
Let $L_k$ denote the subgraph of $L$ consisting of its first $2^k$ rows (namely odd pairs $(a,b)$ with $0\leq a+b<2^k$).
Then $L_1$ is the binary tree \includegraphics[width=0.5cm]{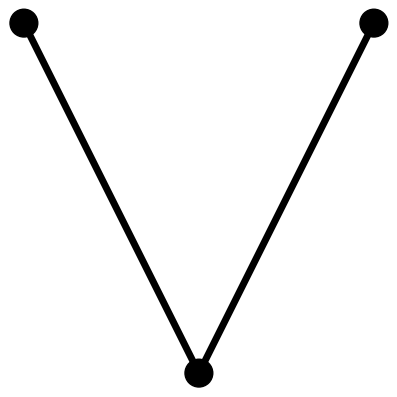}.
All the nodes in $L_1$ are one-dimensional.
Moreover,
\begin{itemize}
\item For each $k\geq 1$, $L_{k+1}$ can be obtained from $L_k$ as follows: to each one-dimensional node in the top row of $L_k$, add an edge and attach one copy of $L_k$.
\item One of the one-dimensional nodes in each copy of $L_k$ so attached becomes a one-dimensional node in the top row of $L_{k+1}$.
\end{itemize}
\begin{figure}
  \begin{center}
    \includegraphics[width=0.7\textwidth]{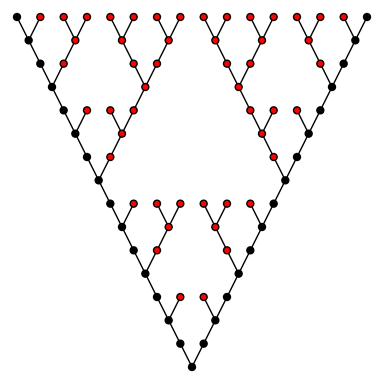}
  \end{center}
  \caption{The binary tree $L_4$. One-dimensional nodes are black.}
  \label{fig:pascal}
\end{figure}
Let $\ell(n)$ be the number of odd pairs summing to $n$ and $\nu(n)$ be the number of $1$'s in the binary expansion of $n$.
The two axioms suffice to construct $L$ by recursion (see Fig.~\ref{fig:pascal}), and imply that, if $2^k\leq n<2^{k+1}$, then $\ell(n)=2 \ell(n-2^k)$.
It then follows that $\ell(n) = 2^{\nu(n)}$.
Our results on Macdonald trees of symmetric and hyperoctahedral groups are more subtle analogs of this phenomenon.

\subsection{Odd partitions of symmetric groups}
\label{sec:odd-sym}
We will need some definitions. A partition $\lambda \vdash n$ is said to be a {\em hook} if $\lambda = (k,1^{n-k})$ for some $k$, $1 \leq k \leq n$. In the special case when $k$ equals $1$ or $n$, we say that the partition is {\em one-dimensional}.

Let $a(n)$ be the number of odd partitions of $n$.
Let $n$ be a positive integer with binary expansion given by
\begin{equation}    \label{eq:1}
  n = \epsilon + 2^{k_1} + 2^{k_2} + \dotsb + 2^{k_r},\,
  \epsilon \in \{0, 1\},\; 0<k_1<k_2<\dotsb <k_r.
\end{equation}
Then we have the following result:
\begin{theorem}[{\cite{macdonald1971degrees,mckay-1972}}]
  \label{thm:macdonald}
  If $n$ is an integer having binary expansion given in \eqref{eq:1},
  then
  \[
    a(n) = 2^{k_1 + \cdots + k_r}.
  \]
\end{theorem}
In other words, if we express $n$ in binary as a sum of powers of $2$, then $a(n)$ is the {\it product} of those powers of $2$.  A recursive version of Theorem~\ref{thm:macdonald} is:
\begin{theorem}
  \label{thm:macdonald-var}
  We have $a(0)=1$, and moreover,
  for any positive integer $n$, let $k$ be such that $2^k\leq n < 2^{k+1}$.
  Then
  \begin{displaymath}
    a(n) = 2^k a(n-2^k).
  \end{displaymath}
\end{theorem}
Our first result is an interpretation of Theorem~\ref{thm:macdonald-var} in terms of Young's lattice.
Let $M$ denote the subgraph induced in Young's graph by the set of odd partitions.
The first thirteen rows of $M$ inside Young's graph are shown in Figure~\ref{fig:odd_in_yng}.
Let $M_k$ denote the first $2^k$ rows of $M$ (consisting of partitions of $0$ through $2^k-1$).
\begin{center}
  \begin{figure}[htbp!]
    \centering
    \includegraphics[width=\textwidth]{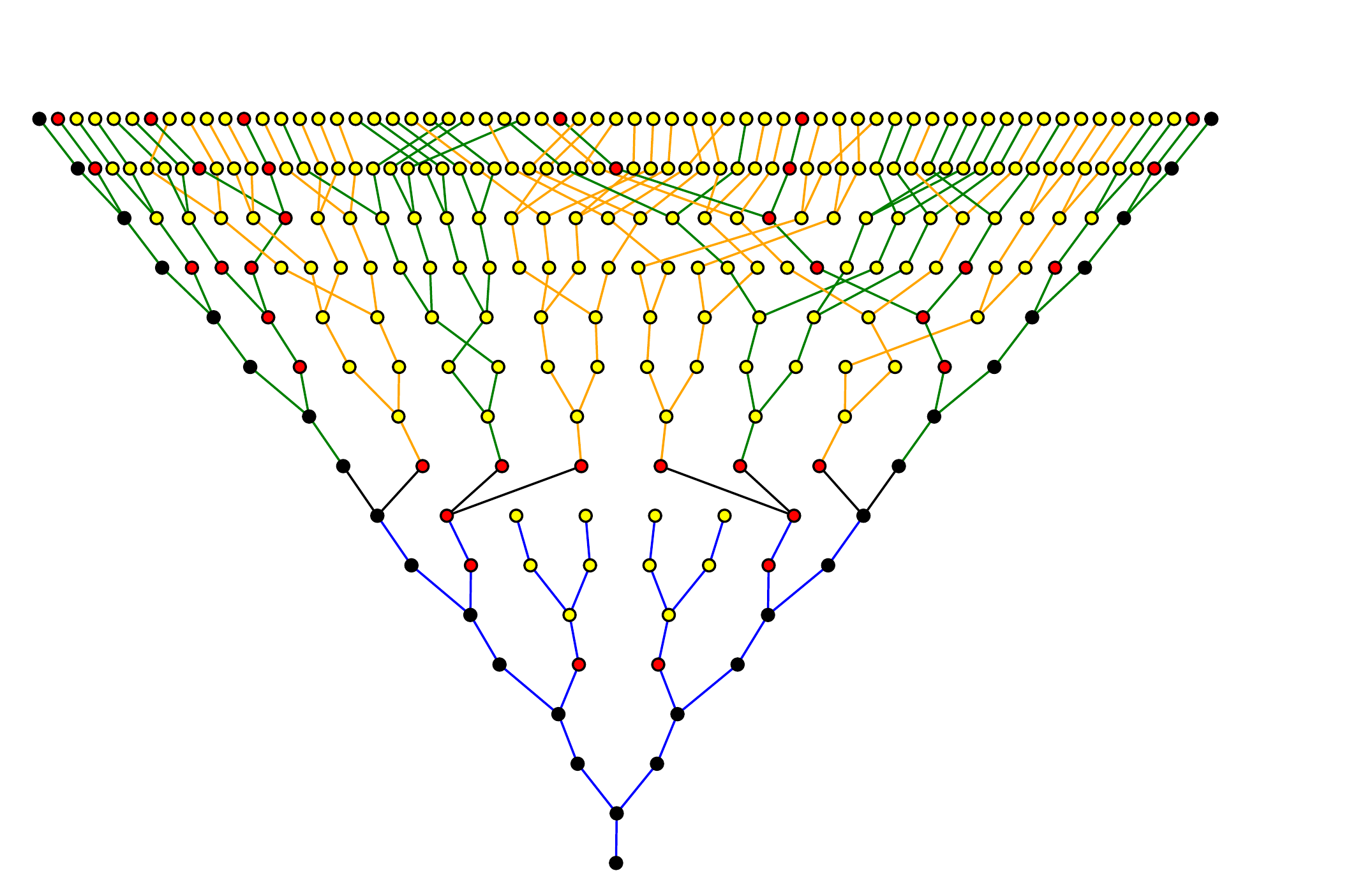}
    \caption{The recursive construction of $M_4$ from $M_3$. $M_3$ is shown in blue. Two copies of $M_3$ (shown in green and orange) are attached to each of the four odd-dimensional hooks of rank seven. The black nodes represent one-dimensional partitions, the red nodes represent higher-dimensional hooks, and the remaining nodes are yellow.}
    \label{fig:recursive}
  \end{figure}
\end{center}
\begin{theorem}
  [Recursive description of $M$]
  \label{theorem:recursive}
  $M_2$ is the rooted tree
  \begin{displaymath}
    \includegraphics[width=1cm]{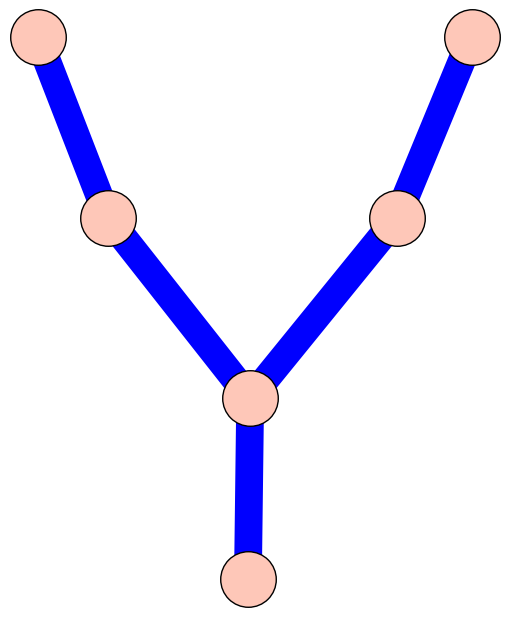}.
  \end{displaymath}
  For each $k\geq 2$, we have:
  \begin{enumerate}
  \item The graph $M_{k+1}$ is obtained from $M_k$ by attaching two copies of
    $M_k$ to each hook in the top row of $M_k$ via an edge. 
  \item The hooks in the top row of $M_{k+1}$ are the hooks in $M_k$ that ascend from the one-dimensional nodes in the top row of $M_k$.
  \item The one-dimensional nodes in the top row of $M_{k+1}$ are one-dimensional nodes in one of the two copies of $M_k$ that
    ascend from the one-dimensional nodes in the top row of $M_k$.
  \end{enumerate}
\end{theorem}
In particular, $M_{k+1}$ is obtained from $M_k$ by attaching copies of $M_k$ to certain leaves.
The entire Macdonald tree $M$ can be constructed by starting with $M_2$ and applying the recursive rules above.
The fact that $M$ is a tree means that an odd-dimensional irreducible representation of $S_n$, when restricted to $S_{n-1}$, contains exactly one odd-dimensional representation of $S_{n-1}$.

\begin{remark}
  The rules in Theorem~\ref{theorem:recursive} are consistent with the following observations:
  \begin{enumerate}
  \item The graph $M_k$ is a rooted binary tree with $2^{\binom k2}$ nodes in its top row.
  \item Of these nodes, $2^{k-1}$ are hooks.
  \item Two of these hooks are one-dimensional.
  \item The subgraph of hooks in $M$ is isomorphic to the subgraph $L$ of odd pairs.
  \end{enumerate}
\end{remark}
Figure~\ref{fig:recursive} shows how $M_4$ is constructed from $M_3$.
Theorem \ref{theorem:recursive} may be regarded as a structural explanation of Theorem~\ref{thm:macdonald-var}, similar to the explanation for the number of odd numbers in the $n$'th row of Pascal's triangle in Section~\ref{sec:baby-case:-pascals}.
The main ingredient in the proof of Theorem~\ref{theorem:recursive} is the following recursive description of odd partitions and their cover relations (see \cite[Lemma~2]{ayyer-prasad-spallone-odd-16}).

\begin{figure}[htbp!]
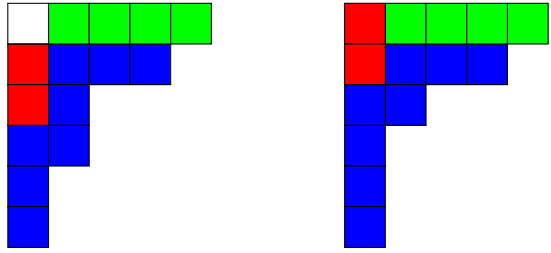

  \begin{center}
    \[
      \begin{array}{c c}
        \ydiagram{1}*[*(red)]{1+0,1,1}*[*(green)]{1+4}*[*(blue)]{1+0,1+3,1+1,2,1,1}
        & \qquad \qquad
          \ydiagram[*(red)]{1,1}*[*(green)]{1+4}*[*(blue)]{1+0,1+3,2,1,1,1}
      \end{array}
    \]
    \caption{The odd partition $\lambda=(5, 4, 2, 2, 1, 1)$ on the left and
      the unique odd partition that it covers, $\mu = (5, 4, 2, 1, 1, 1)$, on the right.
      Both can be reduced to their $2$-cores, $(1)$ and $\emptyset$ respectively, by successively removing unique rim-hooks of size $8$ (blue), $4$ (green) and $2$ (red) with the same leg-lengths.}
    \label{fig:odd-rec}
  \end{center}
\end{figure}
\begin{theorem}
  \label{theorem:odd_rec}
  Let $n$ be a positive integer, and let $k$ be the non-negative integer such that $2^k\leq n<2^{k+1}$.
  A partition $\lambda \vdash n$ is odd if and only if
  \begin{enumerate}
  \item $\lambda$ has a $2^k$-hook, and 
  \item the partition of $n-2^k$ obtained from $\lambda$ by removing this hook (the $2^k$-core of $\lambda$) is odd.
  \end{enumerate}    
  Moreover, the odd partition $\lambda$ is uniquely determined by its $2^k$-core and the leg-length of its $2^k$-hook.
  Finally, if $2^k<n<2^{k+1}$, and $\mu$ is an odd partition of $n-1$, then $\mu$ is covered by $\lambda$ if and only if the $2^k$-hooks of $\lambda$ and $\mu$ have the same leg-length and the $2^k$-core of $\mu$ is covered by the $2^k$-core of $\lambda$.
\end{theorem}

Theorem~\ref{theorem:odd_rec} provides fast algorithms to
sequentially enumerate all odd partitions of $n$ and to generate a uniformly random odd partition of $n$.
See Section~\ref{sec:comparison} for an example.

Repeated applications of Theorem~\ref{theorem:odd_rec} show that any odd partition of $n$ of the form~(\ref{eq:1}) can be reduced to a partition of size $\epsilon$ by successively removing rims of hooks of size $2^{k_r}$, $2^{k_{r-1}}$, and so on. See Figure~\ref{fig:odd-rec} for an example.
\subsection{Odd-dimensional representations of hyperoctahedral groups}
\label{sec:odd-part-hyper}
The $n$'th hyperoctahedral group is the wreath product ${\mb B}_n = S_2\wr S_n = S_2^n\rtimes S_n$. 
The normal subgroup $S_2^n\lhd {\mb B}_n$ has two $S_n$-invariant characters, the trivial character, and the character $\varepsilon$ whose restriction to each factor $S_2$ is its non-trivial character.

For $\lambda \vdash n$, consider the irreducible representations $\rho_\lambda^0$ and $\rho_\lambda^1$ of ${\mb B}_n$ on the representation space $V_\lambda$ of $S_n$ defined by:
\begin{displaymath}
  \rho_\lambda^0(xw) = \rho_\lambda(w) \text{ and } \rho_\lambda^1(xw) = \epsilon(x)\rho_\lambda(w), \text{ for } x\in S_2^n, \text{ and } w\in S_n.
\end{displaymath}
Now suppose that $n = a + b$ for two non-negative integers $a$ and $b$, and that $\alpha$ is a partition of $a$, and $\beta$ a partition of $b$.
The pair $(\alpha,\beta)$ is called a {\em bipartition} of $n$ and we write $(\alpha,\beta) \models n$ in this case.  The number of bipartitions of $n$ is denoted $p_2(n)$ and it is well-known that~\cite[Sequence~A000712]{Sloane}
\[
  \sum_{n \geq 0} p_2(n) x^n = \prod_{j \geq 1} \frac{1}{(1-x^j)^2}.
\]
Consider the subgroup ${\mb B}_a\times {\mb B}_b\cong S_2^n\rtimes (S_a\times S_b)< {\mb B}_n$.
Define
\begin{equation}
  \label{eq:3}
  \rho_{\alpha\beta} = \ind_{{\mb B}_a\times {\mb B}_b}^{{\mb B}_n} \rho_\alpha^0\boxtimes \rho_\beta^1,
\end{equation}
where $\boxtimes$ is the external tensor product.
Then
\begin{displaymath}
  \{\rho_{\alpha\beta}\mid \alpha\vdash a, \beta\vdash b, a+b = n\}
\end{displaymath}
is a complete set of representatives for the set of isomorphism classes of irreducible representations of ${\mb B}_n$.
Moreover, the restriction of $V_{\alpha\beta}$ to ${\mb B}_{n-1}$ contains $V_{\gamma\delta}$ if and only if either $\gamma\in \alpha^-$ and $\delta = \beta$, or $\gamma = \alpha$ and $\delta\in \beta^-$. 
Here $\lambda^-$ is the set of all partitions that are obtained from the Young diagram of $\lambda$ by removing a corner box.
Thus the Bratteli diagram of the family $\mb B_n$ of hyperoctahedral groups turns out to be the Hasse diagram of the poset of bipartitions, partially ordered under containment.
This poset is the \emph{Cartesian product} \cite[Section~3.2]{ec1} of Young's lattice with itself.
We denote this Bratteli diagram by $Y^2$.

We say that a bipartition $(\alpha,\beta)$ is {\em odd} if the dimension of $\rho_{\alpha\beta}$ is odd.
Equation~(\ref{eq:3}) implies that the dimension of $\rho_{\alpha\beta}$ is given by:
\begin{displaymath}
  f_{\alpha\beta} = \frac{n!}{a!b!}f_\alpha f_\beta.
\end{displaymath}
As a consequence, $(\alpha,\beta)$ is an odd bipartition if and only if $\alpha$ and $\beta$ are both odd partitions, and the binomial coefficient $C(a,b)$ is odd.
Let $\bin(n)$ denote the set of place values where $1$ occurs in the binary expansion of $n$.
By Lucas's theorem (see~\cite{fine}, for example), $C(a,b)$ is odd if and only if there exists a subset $S\subset \bin(n)$ such that $\bin(a)=S$ and $\bin(b)=\bin(n) \setminus S$.
These considerations, combined with Macdonald's formula for counting the number of odd partitions, lead to his formula for the number $a_2(n)$ of odd bipartitions of $n$ from \cite{macdonald1973cox}:
\begin{theorem}
  [Macdonald \cite{macdonald1973cox}]
  \label{thm:macdonald2}
  If $n$ is an integer having binary expansion given in \eqref{eq:1},
  then
  \[
    a_2(n) = a(2n).
  \]
\end{theorem}
A recursive version of Theorem~\ref{thm:macdonald2} is:
\begin{theorem}
  \label{theorem:recursive-mac2}
  We have $a_2(0)=1$, and for any positive integer $n$, let $k$ be such that $2^k\leq n<2^{k+1}$.
  Then
  \begin{displaymath}
    a_2(n)=2^{k+1} a_2(n-2^k).
  \end{displaymath}
\end{theorem}
As in the case of symmetric groups, Theorem~\ref{theorem:recursive-mac2} has an explanation in terms of the Bratteli diagram of hyperoctahedral groups.

Let $M^2$ denote the subgraph induced in $Y^2$ by odd bipartitions.
The above description of odd bipartitions easily implies that $M^2$ is also an incomplete binary tree.
Indeed, given an odd bipartition $(\alpha,\beta)$, with $|\alpha|=a$, and $|\beta|=b$, if $n=a+b$ is odd, then exactly one of $a$ and $b$ is odd.
If $a$ is odd, then the unique odd bipartition covered by $(\alpha,\beta)$ in $Y^2$ is $(\gamma,\beta)$, where $\gamma$ is the unique odd partition covered by $\alpha$.
Also, the bijection $(\alpha,\beta)\mapsto (\beta,\alpha)$ is an automorphism of $Y^2$.
We call $M^2$ the \emph{hyperoctahedral Macdonald tree}.
Let $M^2_k$ denote the first $2^k$ rows of $M^2$, consisting of bipartitions with weights $0$ through $2^k-1$.
Call a bipartition $(\alpha,\beta)$ a \emph{hook} if either $\alpha=\emptyset$ or $\beta=\emptyset$, and the non-empty component is a hook.
A bipartition $(\alpha,\beta)$ is said to be \emph{one-dimensional} $f_{\alpha\beta}=1$.

\begin{center}
  \begin{figure}[htbp!]
    \ytableausetup{smalltableaux,aligntableaux=center}
    \includegraphics[width=\textwidth]{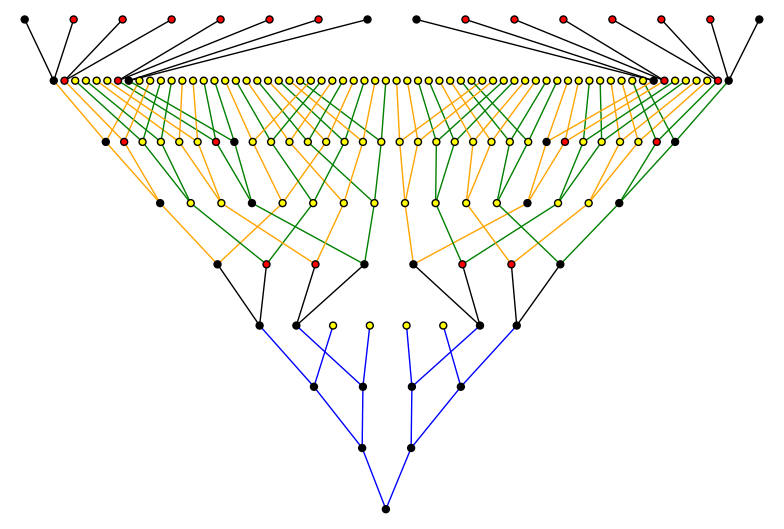}
    \caption{The recursive construction of $M^2_3$ from $M^2_4$.  Two copies of $M^2_2$ (shown in orange and green) are attached to each of the four hook-nodes in $M^2_2$ of rank $3$. The black nodes represent one-dimensional bipartitions, while the red nodes represent higher-dimensional hook bipartitions.}
    \label{fig:hypermac}
  \end{figure}
\end{center}
\begin{theorem}[Recursive description of $M^2$]
  \label{theorem:recursive-hypermac}
  $M^2_2$ is the rooted tree with a distinguished left branch and a right branch.
  \begin{displaymath}
    \includegraphics[width=2cm]{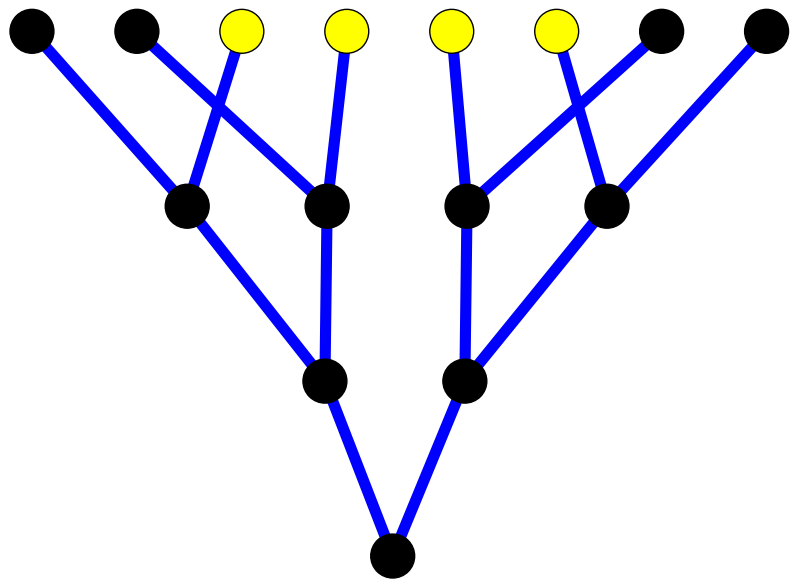}
  \end{displaymath}
  Nodes colored black in the top row of this tree are one-dimensional (and therefore also hooks).
  The remaining nodes in the top row are not hooks.
  For $k\geq 2$ we have:
  \begin{enumerate}
  \item $M^2_{k+1}$ is obtained from $M^2_k$ by attaching two copies of $M^2_k$ 
    to each hook in the top row of $M^2_k$ via an edge.
  \item The hooks in the top row of the left (respectively, right) branch of $M^2_{k+1}$ are the hooks in the left (respectively, right) branch of one copy of $M^2_k$ that ascends from each one-dimensional node in the top row of $M^2_k$.
  \item The one-dimensional nodes in the left (respectively, right) branch of the top row of $M^2_{k+1}$ are one of the two one-dimensional nodes in in the left (respectively, right) branch of one copy of $M^2_k$ that ascends from a one-dimensional node of the left (respectively, right) branch of $M^2_k$.
  \end{enumerate}
\end{theorem}
All of $M^2$ can be reconstructed from $M^2_2$ by recursively applying these rules. The fact that $M^2$ is a tree means that an odd-dimensional irreducible representation of $\mathbb{B}_n$, when restricted to $\mathbb{B}_{n-1}$, contains exactly one odd-dimensional representation of $\mathbb{B}_{n-1}$.

\begin{remark}
  The recursive rules in Theorem~\ref{theorem:recursive-hypermac} are consistent with the following observations:
  \begin{enumerate}
  \item $M^2_k$ is a rooted binary tree with $2^{\binom{n+1}2}$ nodes in its topmost row.
  \item Of these, $2^k$ nodes are hooks. $2^{k-1}$ of these hooks are on the left branch, and $2^{k-1}$ of them are on the right branch.
  \item Of these hooks, $4$ are one-dimensional; $2$ on the left branch, and $2$ on the right branch.
  \end{enumerate}
\end{remark}

Theorem~\ref{theorem:recursive-hypermac} can be proved using a characterization of vertices and edge relations in $M^2$ (the analog of Theorem~\ref{theorem:odd_rec}).

\begin{center}
  \begin{figure}
    \[
      \begin{array}{c}
        ((11,2^3,1^2),(4^2,1^4)) = \left(
        \ydiagram{1}*[*(blue)]{1+0,1,1}*[*(red)]{1+10,1+1,1+1,2,1,1}
        ,
        \ydiagram[*(yellow)]{4}*[*(green)]{4+0,4,1,1,1,1}\right)\\
        \\
        ((11,2^2,1^3),(4^2,1^4)) = \left(\ydiagram[*(blue)]{1,1}*[*(red)]{1+10,1+1,2,1,1,1}
        ,
        \ydiagram[*(yellow)]{4}*[*(green)]{4+0,4,1,1,1,1}\right)\\
      \end{array}
    \]
    \caption{The odd bipartition in the second row is the unique odd bipartition covered by the odd bipartition in the first row.}
    \label{fig:hyper-odd-rec}
  \end{figure}
\end{center}
\begin{theorem}
  \label{theorem:hyper-odd-rec}
  Let $n$ be a positive integer and let $k$ be the unique non-negative integer such that $2^k\leq n<2^{k+1}$.
  A bipartition $(\alpha,\beta)$ of $n$ is odd if and only if either $\alpha$ or $\beta$ has a hook of length $2^k$, and the bipartition $(\tilde\alpha,\tilde\beta)$ obtained after removing the rim of this hook is an odd bipartition of $n-2^k$.
  Moreover the odd bipartition $(\alpha,\beta)$ is uniquely determined by $(\tilde\alpha,\tilde\beta)$, the leg-length of the hook which was removed, and the knowledge of whether it was removed from $\alpha$ or from $\beta$.
\end{theorem}
\subsection{Demihyperoctahedral groups}
\label{sec:odd-demi}

We denote by $\mb D_n$ the kernel of $\eps: \mb B_n \to \{ \pm 1\}$, where $\eps$ is the character defined at the beginning of Section~\ref{sec:odd-part-hyper}.
It is the Weyl group of type $D_n$, and called the {\em demihyperoctahedral group}.
The irreducible representations of $\mb D_n$ are of two kinds:
\begin{itemize}
\item Let $(\alpha,\beta)$ be a bipartition of $n$ such that $\alpha\neq\beta$.
  Then the irreducible representation $\rho_{\alpha\beta}$ of $\mb B_n$ remains irreducible when restricted to $\mb D_n$, and is isomorphic to the restriction of $\rho_{\beta\alpha}$ to $\mb D_n$.
\item Let $n>0$ be even, and let $\alpha$ be any partition of $n/2$.
  Then the irreducible representation $\rho_{\alpha\alpha}$ of $\mb B_n$, when restricted to $\mb D_n$, is a sum of two non-isomorphic irreducible representations of $\mb D_n$, which are denoted by $\rho_{\alpha\alpha}^+$ and $\rho_{\alpha\alpha}^-$.
\end{itemize}
The representation $\rho_{\alpha\alpha}^-$ is the twist of $\rho_{\alpha\alpha}^+$ by any element $e \in \mathbb{B}_n \setminus \mathbb{D}_n$.
For $n>0$, we have
\begin{displaymath}
  \Irr(\mb D_n)= \{ \Res^{\mathbb{B}_n}_{\mathbb{D}_n} \rho_{\alpha\beta} \mid (\alpha,\beta) \models n, \alpha \neq \beta \} \coprod \left\{ \rho_{\alpha\alpha}^{+}, \rho_{\alpha\alpha}^- \mid \alpha \vdash \frac{n}{2} \right\},
\end{displaymath}
with the second set nonempty only when $n$ is even.  The total number of representations of $\mathbb{D}_n$ is therefore equal to $\half p_2(n)$ when $n$ is odd, and equal to
\begin{displaymath}
  \half p_2(n)+ \frac{3}{2} p\left(\frac{n}{2} \right)
\end{displaymath}
when $n$ is even.
The branching rules are as follows: for every partition $\alpha$ that is covered by a partition $\beta$ in Young's lattice,
\begin{itemize}
\item $\rho_{\alpha\alpha}^\pm$ occurs in the restriction of $\rho_{\alpha\beta}$.
\item $\rho_{\alpha\beta}$ occurs in the restrictions of $\rho_{\beta\beta}^\pm$.
\item $\rho_{\alpha\gamma}$ occurs in the restriction of $\rho_{\beta\gamma}$ for every partition $\gamma$.
\end{itemize}
Let $Y^2_D$ denote the resulting Bratteli diagram.
It turns out that the subgraph induced in $Y^2_D$ by odd-dimensional representations is no longer a tree.
For example,
\begin{displaymath}
  \xymatrix{
    & ((1),(1,1)) \ar@{-}[dl]\ar@{-}[dr] &\\
    ((1),(1))^+ & & ((1),(1))^-\\
    & ((1),\emptyset) \ar@{-}[ul]\ar@{-}[ur]&
  }
\end{displaymath}
is a cycle.
\section{Determinants of representations of Coxeter groups}
\label{sec:chiral-coxeter}
In this section, we collect all the results about irreducible representations of Coxeter groups for which the determinant homomorphism is nontrivial. The results for the symmetric group are presented in Section~\ref{sec:amrich-part} and summarize the work in Ref.~\cite{APS-chiral}. 
The Solomon principle is explained in Section~\ref{sec:SolomonP}.
This is used to explain the results for all other Coxeter groups, which are taken from Ref.~\cite{Ghosh}, and appear in later subsections.
\subsection{Symmetric groups}
\label{sec:amrich-part}
Recall that a partition $\lambda$ of $n$ is said to be {\em chiral} if the corresponding irreducible representation $V_\lambda$ of $S_n$ has non-trivial determinant. 
Let $b(n)$ denote the number of chiral partitions of $n$.
The first few entries of the sequence $b(n)$ \cite[Sequence~\texttt{A272090}]{Sloane} are
\begin{displaymath}
  0, 1, 2, 3, 5, 4, 8, 12, 20, 8, 16, 24, 40, 32, 64, 88, 152, 16, 32, 48,\dotsc.
\end{displaymath}
The numbers $b(n)$ also have a simple combinatorial interpretation.
In Young's seminormal form, $V_\lambda$ has a basis indexed by standard
tableaux of shape $\lambda$. The basis vector
corresponding to a tableau $T$ is an eigenvector for $\rho_\lambda(s_1)$;
the eigenvalue is $-1$ if $2$ occurs in the first column of $T$ and $+1$
otherwise.
Thus $\lambda$ is chiral if and only if $g_\lambda$, the number of standard tableaux
where $2$ occurs in the first column, is odd.
Let
\begin{displaymath}
  C(\lambda) = \sum_{(i,j)\in \lambda} (j-i),
\end{displaymath}
be the sum of contents of cells in the Young diagram of $\lambda$.

\begin{theorem}[{Stanley~\cite[Exercise~7.55(a)]{ec2}}]
  Given a partition $\lambda$ of $n$, we have
  \begin{equation}
    \label{eq:5}
    g_\lambda = \frac{f_\lambda\left(\binom n2 - C(\lambda)\right)}{2\binom n2}.
  \end{equation}
  \label{theorem:chirality-gla}
  The partition $\lambda$ of $n$ is chiral if and only if $g_\lambda$ is odd.
\end{theorem}
We now state the main result of Ref.~\cite{APS-chiral}.
\begin{theorem}
  \label{theorem:all}
  If $n$ is an integer having binary expansion given in \eqref{eq:1},
  then the number of chiral partitions of $n$ is given by
  \begin{displaymath}
    b(n)=a(n)\left(\half  +  \sum_{v=1}^{k_1-1} 2^{v k_1 - \frac{v^2+3v+4}2} + \epsilon \cdot 2^{\binom{k_1}{2}-k_1} \right).
  \end{displaymath}
\end{theorem}
For any integer $m$, let $v_2(m)$ denote the {\em $2$-adic valuation} of $m$, that is, the largest integer $v$ such that $2^v$ divides $m$.
Theorem~\ref{theorem:all} is a consequence of the following more refined enumerative result.
\begin{theorem}
  \label{theorem:refined}
  If $n$ is an integer having binary expansion (\ref{eq:1}),
  then the number $b_v(n)$ of chiral partitions $\lambda$ of $n$ for which $v_2(f_\lambda) = v$ is given by
  \begin{displaymath}
    b_v(n) = 2^{k_2+\dotsb+k_r}\times
    \begin{cases}
      2^{k_1-1} & \text{if } v =0,\\
      2^{(v+1)(k_1-2)-\binom v2} &\text{if } 0<v<k_1,\\
      \epsilon 2^{\binom{k_1}2} &\text{if } v = k,\\
      0 &\text{if } v> k_1.
    \end{cases}
  \end{displaymath}
\end{theorem}
Theorem~\ref{theorem:refined} follows from the characterization of chiral partitions in Theorem~\ref{theorem:chiral-towers}.
The relevant combinatorial data for this characterization is the notion of {\em $2$-core towers} of partitions, which is roughly the following. To each partition $\lambda \vdash n$, we associate a binary tree, each node of which is a partition. Each of these partitions is a $2$-core.
The number of non-trivial rows in the $2$-core tower of $\lambda$ is at most the number of digits in the binary expansion of $n$.
We refer to Refs~\cite{APS-chiral,olsson} for the details.
Let $w_i(\lambda)$ denote the sum of the sizes of the partitions in the $i$'th row of the $2$-core tower of $\lambda$.
We then have
\begin{displaymath}
  n = \sum_{i=0}^\infty w_i(\lambda)2^i.
\end{displaymath}
Let $\nu(n)$ denote the number of times $1$ occurs in the binary expansion of $n$ (for $n$ as in (\ref{eq:1}), $\nu(n) = r+\epsilon$).
Also let $w(\lambda) = \sum_{i\geq 0} w_i(\lambda)$.
Define the \emph{$2$-deviation} of $\lambda$ as
\begin{displaymath}
  e_2(\lambda) = w(\lambda)-\nu(|\lambda|).
\end{displaymath}
The following result gives a formula for the $2$-adic valuation $v_2(f_\lambda)$.

\begin{theorem}[{\cite[Proposition~6.4]{olsson}}]
  \label{theorem:oddness}
  For any partition $\lambda$, $v_2(f_\lambda) = e_2(\lambda)$.
\end{theorem}
In particular, a partition is odd if and only if its $2$-core tower has at most one cell in each row.
Theorem~\ref{theorem:oddness} characterizes partitions parameterizing representations with dimensions of specified $2$-adic valuation in terms of their $2$-core towers.
\begin{theorem}[{\cite[Theorem 6]{APS-chiral}}]
  \label{theorem:chiral-towers}
  Let $\la$ be a partition of $n = a+b$ written as in \eqref{eq:1} with $2$-quotient $(\alpha,\beta)$ with $|\alpha|=a$ and $|\beta|=b$.
  Then $\lambda$ is chiral if and only if one of the following holds:
  \label{chiral}
  \begin{enumerate}
  \item  \label{odd-chiral} $\la$ is odd, and
    \begin{enumerate}
    \item if $n$ is even, then $k_1 \in \bin(a)$.
    \item if $n$ is odd, then $k_1 \in \bin(b)$.
    \end{enumerate}
  \item $\core{\la}{2} =\emptyset$ or $(1)$, and
    \begin{enumerate}
    \item $\alpha$ and $\beta$ are odd,
    \item $\bin(a) \cap \bin(b)= \{ j \}$, with $j=v_2(a)=v_2(b)$.
    \end{enumerate}
  \item $\core{\la}{2}=(2,1)$ and $\phi(\alpha, \beta)$ is odd.
  \end{enumerate}
\end{theorem}
In particular, one can infer whether a partition is chiral or not by looking at its $2$-core tower. From (\ref{odd-chiral}), it follows that half of the odd partitions are chiral and half are achiral. 

Not only does Theorem~\ref{theorem:chiral-towers} prove the enumerative results in Theorems~\ref{theorem:all} and~\ref{theorem:refined}, it also provides fast algorithms to:
\begin{itemize}
\item sequentially enumerate all chiral partitions of $n$ with given $v_2(f_\lambda)$, and
\item generate a uniformly random chiral partition of $n$ with given $v_2(f_\lambda)$.
\end{itemize}
See Section~\ref{sec:comparison} for an example.

The proof of Theorem~\ref{theorem:chiral-towers} takes as its starting point the \emph{Solomon principle}, which we now discuss in the context of a general Coxeter group.
\subsection{Solomon principle} 
\label{sec:SolomonP}
Let $\rho$ be a representation of a Coxeter group $W$.  In this section we show how to infer $\det \rho$ from its character.  So let $(W,S)$ be a Coxeter group ($S$ is a certain set of generators of order $2$; see \cite{Bou.Lie.4-6}).   There is a unique multiplicative character $\eps_W$ so that $\eps_W(s)=-1$ for each $s \in S$, namely
\begin{displaymath}
  \eps_W(w)=(-1)^{\ell(w)},
\end{displaymath}
where $\ell(w)$ is the length of $w$ with respect to $S$.

For the Coxeter groups of type $A_n$, $D_n$, $E_6$, $E_7$, $E_8$, $H_3$, $H_4$, and $I_2(p)$ with $p$ odd, the trivial character and $\eps_W$ are the only multiplicative characters.  This is equivalent to the abelianization $W_{\ab}$ having order  $2$.
\begin{proposition} 
  \label{solomon.principle} 
  Suppose $|W_{\ab}|=2$ , and let $s \in S$.  If $\rho$ is a representation of $W$, then
  \begin{displaymath}
    \det \rho=(\eps_W)^b,
  \end{displaymath}
  where
  \begin{displaymath}
    b=\frac{\dim \rho-\chi_\rho(s)}{2}.
  \end{displaymath}
\end{proposition}

\begin{proof} Let $a$ be the multiplicity of $1$ as an eigenvalue of $\rho(s)$, and $b$ be the multiplicity of $-1$.  Then $\dim \rho=a+b$ and $\chi_\rho(s)=a-b$.
\end{proof}

We attribute this approach to counting chiral partitions to L. Solomon; see \cite[Exercise 7.55]{ec2}.

The abelianization of the other Coxeter groups ($B_n$, $F_4$, and $I_2(p)$ for $p$ even) are Klein $4$ groups.  For these, fix two non-conjugate simple reflections $s_1,s_2 \in S$, and multiplicative characters $\omega_1,\omega_2$ so that $\omega_1(s_1)=-1$, $\omega_1(s_2)=1$, $\omega_2(s_1)=1$, $\omega_2(s_2)=-1$.  Then $\eps_W=\omega_1 \cdot\omega_2$ and the multiplicative characters of $W$ are $\{ 1, \omega_1,\omega_2, \omega_1 \cdot\omega_2 \}$.  A similar proof to that of Proposition \ref{solomon.principle} gives:
\begin{proposition} 
  \label{solomon.principle2} 
  Suppose $|W_{\ab}|=4$, and let $s_1,s_2,\omega_1,\omega_2$ be as above.  If $\rho$ is a representation of $W$, then
  \begin{displaymath}
    \det \rho=(\omega_1)^{x_1}(\omega_2)^{x_2},
  \end{displaymath}
  where
  \begin{displaymath}
    x_1=\frac{\dim \rho-\chi_\rho(s_1)}{2},
  \end{displaymath}
  and
  \begin{displaymath}
    x_2=\frac{\dim \rho-\chi_\rho(s_2)}{2}.
  \end{displaymath}
\end{proposition}
\subsection{Hyperoctahedral groups}
\label{sec:hyper.count}
Recall that $\eps$ denotes the character of ${\mb B}_n = S_2^n\rtimes S_n$ whose restriction to each factor $S_2$ is its non-trivial character.
Write $\sgn^0$ for the composition of the projection $\mb B_n \to S_n$ with the sign character of $S_n$.
Finally write $\sgn^1=\eps \cdot \sgn^0$.
The multiplicative characters of $\mb B_n$ are then $1, \eps, \sgn^0$, and $\sgn^1$.
For $\omega$ a multiplicative character of $\mb B_n$, put
\begin{displaymath}
  N_\omega(n)= \# \{ (\alpha,\beta) \models n \mid \det \rho_{\alpha,\beta}=\omega \}.
\end{displaymath}

In this section we will give closed formulas for $N_\omega(n)$ when $\omega \neq 1$.  (If we had a closed formula for $N_1(n)$, then we would have a closed formula for $p_2(n)$, which seems out of reach.)

The determinants of the representations of $\mb B_n$ can be computed either through the Solomon formulas, combined with the Frobenius character formula, or
through a well-known formula for the determinant of an induced representation.  (See \cite{Ghosh}, for example.)
Let
\begin{displaymath}
  x_{\alpha \beta}=f_\alpha f_\beta \binom{n-1}{a,b-1},
\end{displaymath}
and
\begin{displaymath}
  y_{\alpha \beta}=f_\alpha f_\beta \binom{n-2}{a-1,b-1}+ f_\beta g_\alpha \binom{n-2}{a-2,b}+ f_\alpha g_\beta \binom{n-2}{a,b-2}.
\end{displaymath}

\begin{theorem} \label{hyper.Solomon}
  For a bipartition $(\alpha,\beta)$, we have
  \begin{displaymath}
    \det \rho_{\alpha \beta}=\eps^{x_{\alpha \beta}} \cdot (\sgn^0)^{y_{\alpha \beta}}.
  \end{displaymath}
\end{theorem}

Since the parities of $f_\lambda$ and $g_\lambda$ can be read from the $2$-core tower of $\lambda$ by Theorems~\ref{theorem:oddness} and ~\ref{theorem:chiral-towers}, the parities of $x_{\alpha \beta}$ and $y_{\alpha \beta}$ are also determined by the same tower.  In particular, this allows for closed formulas for
\begin{displaymath}
  N_\omega(a,b)= \# \{ \alpha \vdash a, \beta \vdash b \mid \det \rho_{\alpha,\beta}=\omega \},
\end{displaymath}
at least when $\omega$ is nontrivial.
From this, the $N_\omega(n)$ are computed by the identity
\begin{displaymath}
  N_\omega(n)=\sum_{a+b=n} N_\omega(a,b).
\end{displaymath}
Here is the result:
For $n \geq 2$, and $k=v_2(n)$, we have
\begin{equation}
  \label{N_eps(n)}
  N_{\eps}(n)= \begin{cases}
    \frac{1}{4} a_2(n)   &\text{if $n$ is odd}, \\
    \frac{1}{8}a_2(n)  \left(2+\sum_{j=1}^{k-1}2^{\binom{k}{2}-\binom{j}{2}} \right)  & \text{if $n$ is even}.\\
  \end{cases}
\end{equation}

For $n$ as in (\ref{eq:1}), let $n'=n-\epsilon$.
The following equations compute $N_{\sgn^0}(n)$ in all cases~\cite{Ghosh}.  
Let $k=v_2(n')$.
\label{N_sgn0(n)}
\begin{enumerate}
\item If  $n \equiv 1 \mod 4$, then
  \begin{displaymath}
    N_{\sgn^0}(n) =\frac{1}{4} a_2(n) \left( 1+3 \cdot 2^{\binom{k}{2}-1}+2^{\binom{k}{2}-k+1}+ \sum_{j=2}^{k-1} \left(2^{\binom{k}{2}-\binom{j}{2}}+2^{\binom{k}{2}-j} \right)\right).
  \end{displaymath}
\item If  $n \equiv 3 \mod 4$, then   $ N_{\sgn^0}(n) =  \half a_2(n)$.
\item If  $n$ is even, then
  \begin{displaymath}
    N_{\sgn^0}(n) =   \frac{1}{8}a_2(n)  \left(2+\sum_{j=1}^{k-1}2^{\binom{k}{2}-\binom{j}{2}} \right).
  \end{displaymath}
\end{enumerate}
For $n \geq 2$, and $k=v_2(n)$, we have
\begin{equation*}
  N_{\sgn^1}(n)= \begin{cases}
    \frac{1}{4}a_2(n) &\text{if } n \text{ is odd} \\
    \frac{1}{8}a_2(n) \left( 2+2^k+ \sum_{j=1}^{k-1} 2^{\binom{k}{2}-\binom{j}{2}}(2^j-1) \right) & \text{if $n$ is even}.\\
  \end{cases}
\end{equation*}
Since $N_{1}(n)+ N_{\eps}(n)+N_{\sgn^0}(n)+N_{\sgn^1}(n)=p_2(n)$,
one implicitly has a formula for $N_1(n)$ as well.
Figure~\ref{fig:Bn-plot} shows a logplot of each $N_\omega(n)$, with $N_1(n)$ in orange,  $N_{\sgn^0}(n)$ in green, $N_{\sgn^1}(n)$ in red, and $N_{\eps}(n)$ in blue.
Figure~\ref{fig:Bn-plot} suggests the following inequalities, which have been shown to hold for $n \geq 10$:
\begin{enumerate}
\item  $N_{\eps}(n)=N_{\sgn^1}(n)<N_{\sgn^0}(n)<N_{1}(n)$, for $n$ odd
\item $N_{\eps}(n)=N_{\sgn^0}(n)<N_{\sgn^1}(n)<N_{1}(n)$, for $n$ even.
\end{enumerate}
\begin{center} 
  \begin{figure}
    \includegraphics[width=\textwidth]{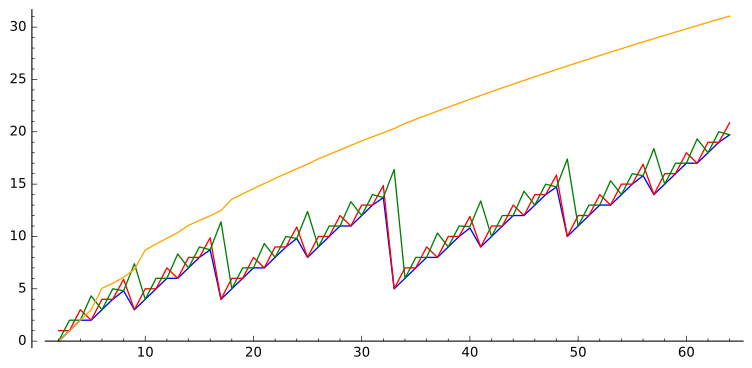}
    \caption{$N_1(n)$ (orange), $N_{\sgn^0}(n)$ (green), $N_{\sgn^1}(n)$ (red) and $N_{\eps}(n)$ (blue).}
    \label{fig:Bn-plot}
  \end{figure}
\end{center}
\subsection{Demihyperoctahedral groups}
\label{demi.count}

The group $\mb D_n$, for $n \geq 2$, has two multiplicative characters: $1$ and $\sgn$, where $\sgn$ is the restriction of $\sgn^0$ (or of $\sgn^1$) to $\mb D_n$.  To avoid confusion with earlier notation, let us write
\begin{displaymath}
  N_\omega'(n)= \# \{ \rho \in \Irr(\mb D_n) \mid \det \rho=\omega \},
\end{displaymath}
for $\omega=1,\sgn$.
Since the sum of $N_1'(n)$ and $N_{\sgn}'(n)$ is equal to this number, so we need only count $N_{\sgn}'(n)$. 

The restrictions of $\rho_{\alpha\beta}$ with $\alpha \neq \beta$ to $\mb D_n$ are irreducible, and their determinants are given by restriction.
On the other hand, the restrictions of $\rho_{\alpha\alpha}$ split into two irreducible pieces $\rho_{\alpha\alpha}^{\pm}$  with identical determinants.  To compute these determinants, we apply the Solomon formula, using that
$\chi_{\rho_{\alpha\alpha}^+}(s_1)=\half \chi_{\rho_{\alpha\alpha}}(s_1)$.  Unless $n=2^k$ or $n=2^k+2$ for some $k$, one gets the trivial character.  If $n$ is a power of $2$, then $\rho_{\alpha\alpha}^{\pm}$ has determinant $\sgn$ if and only if $\alpha$ is odd and achiral.  If $n=2^k+2$ for some $k$, then $\rho_{\alpha\alpha}^{\pm}$ has determinant $\sgn$ if and only if $\alpha$ is odd.  This leads to the following count:
\begin{theorem} 
  Let $n \geq 4$.
  We have
  \begin{displaymath}
    N_{\sgn}'(n)=
    \left\{
      \begin{array}{ll}
        \half( N_{\sgn^0}(n)+N_{\sgn^1}(n)) +\half  n&\text{ if $n=2^k$ for some $k\geq1$,}\\
        \half( N_{\sgn^0}(n)+N_{\sgn^1}(n))+n-2 & \text{ if $n=2^k+2$ for some $k\geq 1$},\\
        \half( N_{\sgn^0}(n)+N_{\sgn^1}(n)) & \text{ otherwise.}
      \end{array}
    \right.
  \end{displaymath}
\end{theorem}
Formulas for $N_\omega(n)$ were given in Section~\ref{sec:hyper.count}.
\subsection{Other Coxeter groups}
The remaining Coxeter groups are either dihedral or exceptional.  The dihedral case is facile, and the exceptional cases are finite in number and the determinants may
be computed by available character tables.  See \cite{Ghosh} for these remaining cases.
\section{Comparison of odd and chiral partitions}
\label{sec:comparison}
It turns out that the functions $a(n)$ and $b(n)$ track each other closely (see Figure~\ref{fig:growth}).
\begin{figure}[h]
  \centering
  \includegraphics[scale=0.6]{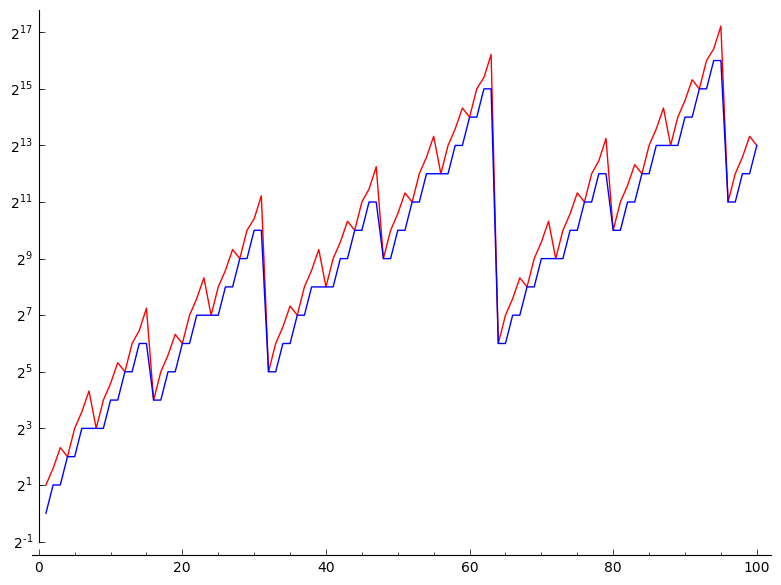}
  \caption{Growth of \textcolor{blue}{$a(n)$} (odd partitions) against \textcolor{red}{$b(n+2)$} (chiral partitions)  on a log-linear scale.}
  \label{fig:growth}
\end{figure}
\begin{theorem}
  \label{theorem-an-bn+2-comparison}
  For every positive integer $n$,
  \begin{displaymath}
    \frac{2}{5} \leq \frac{a(n)}{b(n+2)} \leq 1.
  \end{displaymath}
  Moreover, $a(n)/b(n+2) = 1$ if and only if $n$ is divisible by $4$.
\end{theorem}
See \cite{APS-chiral} for the proof.
Theorem \ref{theorem-an-bn+2-comparison} says that $a(n)$ is a good proxy for estimating the growth of $b(n)$.
The order of the sequence $a(n)$ fluctuates widely; when $n=2^k$, $a(n) = 2^k$ and when $n=2^k-1$, $a(n) = 2^{k(k-1)/2}$.
In any case, $a(n)$ (and thus also $b(n)$) is dwarfed by the growth of the partition function:
\begin{displaymath}
  p(n) \sim \frac 1{4n\sqrt 3} \exp(\pi \sqrt{2n/3}) \text{ as } n\to \infty.
\end{displaymath}
For example, Theorem~\ref{theorem:all} predicts a relatively large value for $b(4097)$ (compared to neighboring integers), but even so, the probability of a partition of $4097$ being chiral is
\begin{displaymath}
  b(4097)/p(4097) \approx 2.4236148775415833\times 10^{-47},
\end{displaymath}
which is astronomically small.
Therefore, algorithms for sampling uniformly random partitions will be very ineffective for sampling uniformly random odd and chiral partitions.
But using Theorem~\ref{theorem:odd_rec} and Theorem~\ref{theorem:chiral-towers}, one may easily generate these of size $4097$ very fast.
A sample run of our code, which is available at \url{http://www.imsc.res.in/\~amri/chiral.sage},
gives random odd and chiral partitions of $4095$ and $4097$ respectively instantaneously on a conventional office desktop:
\begin{lstlisting}
  sage: random_odd_dim_partition(4095).frobenius_coordinates()
  ([677, 491, 148, 65, 24, 6, 2, 1, 0], 
  [1556, 446, 346, 206, 107, 7, 3, 1, 0])

  sage: random_chiral_partition(4097).frobenius_coordinates()
  ([1879, 272, 152, 27, 20, 19, 8, 2, 0], 
  [1015, 239, 168, 103, 100, 43, 32, 7, 2])
\end{lstlisting}
These programs also allow us to enumerate odd and chiral partitions efficiently.
Our code for chiral partitions based on Theorem~\ref{theorem:chiral-towers} also provides algorithms for generating random chiral partitions with dimension having fixed $2$-adic valuation, and for enumerating all chiral partitions of $n$ with dimension having fixed $2$-adic valuation.
\section*{Acknowledgements}
We thank A. R. Miller for pointing us to McKay's work.
We thank Debarun Ghosh for providing us with Figure~\ref{fig:Bn-plot}.
This research was driven by computer exploration using the open-source
mathematical software \texttt{Sage}~\cite{sage} and its algebraic
combinatorics features developed by the \texttt{Sage\-Combinat}
community~\cite{Sage-Combinat}. AA was partially supported by UGC centre for Advanced Study grant and by Department of Science and Technology grant EMR/2016/006624. AP was partially supported by a Swarnajayanti Fellowship of the Department of Science and Technology (India).
\bibliographystyle{plain}
\bibliography{refs}
\end{document}